\documentclass[a4paper,12pt]{amsart}
\usepackage{ogonek}
\usepackage{amsmath,amsfonts,amssymb,amsthm,amscd}
\usepackage{latexsym,graphicx}
\usepackage[matrix,arrow,ps,color,line,curve,frame]{xy}
\usepackage[usenames,dvipsnames]{color}
\usepackage{mathrsfs}
\usepackage{tikz}
\usetikzlibrary{arrows}
\usetikzlibrary{decorations.pathreplacing,decorations.markings}

\renewcommand{\[}{\begin{equation}}
\renewcommand{\]}{\end{equation}}

\begin{document}
\baselineskip15pt
\newcommand{\nc}[2]{\newcommand{#1}{#2}}
\newcommand{\rnc}[2]{\renewcommand{#1}{#2}}
\rnc{\theequation}{\thesection.\arabic{equation}}
\def\note#1{{}}

\newtheorem{definition}{Definition $\!\!$}[section]
\newtheorem{proposition}[definition]{Proposition $\!\!$}
\newtheorem{lemma}[definition]{Lemma $\!\!$}
\newtheorem{corollary}[definition]{Corollary $\!\!$}
\newtheorem{theorem}[definition]{Theorem $\!\!$}
\newtheorem{example}[definition]{\sc Example $\!\!$}
\newtheorem{remark}[definition]{\sc Remark $\!\!$}

\nc{\beq}{\begin{equation}}
\nc{\eeq}{\end{equation}}
\rnc{\[}{\beq}
\rnc{\]}{\eeq}
\nc{\ba}{\begin{array}}
\nc{\ea}{\end{array}}
\nc{\bea}{\begin{eqnarray}}
\nc{\beas}{\begin{eqnarray*}}
\nc{\eeas}{\end{eqnarray*}}
\nc{\eea}{\end{eqnarray}}
\nc{\be}{\begin{enumerate}}
\nc{\ee}{\end{enumerate}}
\nc{\bd}{\begin{diagram}}
\nc{\ed}{\end{diagram}}
\nc{\bi}{\begin{itemize}}
\nc{\ei}{\end{itemize}}
\nc{\bpr}{\begin{proposition}}
\nc{\bth}{\begin{theorem}}
\nc{\ble}{\begin{lemma}}
\nc{\bco}{\begin{corollary}}
\nc{\bre}{\begin{remark}\rm}
\nc{\bex}{\begin{example}}
\nc{\bde}{\begin{definition}}
\nc{\ede}{\end{definition}}
\nc{\epr}{\end{proposition}}
\nc{\ethe}{\end{theorem}}
\nc{\ele}{\end{lemma}}
\nc{\eco}{\end{corollary}}
\nc{\ere}{\hfill\mbox{$\Diamond$}\end{remark} }
\nc{\eex}{\hfill\mbox{$\Diamond$}\end{example}}
\nc{\bpf}{{\it Proof.~~}}
\nc{\epf}{\hfill\mbox{$\square$}\vspace*{3mm}}
\nc{\hsp}{\hspace*}
\nc{\vsp}{\vspace*}
\newcommand{\wegdamit}[1]{}
\newcommand{\co}{\,\mathrm{co}\,}

\newcommand{\counit}{\varepsilon}
\newcommand{\ls}[1]{\ell(#1)^{\langle 1\rangle}}
\newcommand{\rs}[1]{\ell(#1)^{\langle 2\rangle}}
\nc{\ot}{\otimes}
\nc{\te}{\!\ot\!}
\nc{\bmlp}{\mbox{\boldmath$\left(\right.$}}
\nc{\bmrp}{\mbox{\boldmath$\left.\right)$}}
\nc{\LAblp}{\mbox{\LARGE\boldmath$($}}
\nc{\LAbrp}{\mbox{\LARGE\boldmath$)$}}
\nc{\Lblp}{\mbox{\Large\boldmath$($}}
\nc{\Lbrp}{\mbox{\Large\boldmath$)$}}
\nc{\lblp}{\mbox{\large\boldmath$($}}
\nc{\lbrp}{\mbox{\large\boldmath$)$}}
\nc{\blp}{\mbox{\boldmath$($}}
\nc{\brp}{\mbox{\boldmath$)$}}
\nc{\LAlp}{\mbox{\LARGE $($}}
\nc{\LArp}{\mbox{\LARGE $)$}}
\nc{\Llp}{\mbox{\Large $($}}
\nc{\Lrp}{\mbox{\Large $)$}}
\nc{\llp}{\mbox{\large $($}}
\nc{\lrp}{\mbox{\large $)$}}
\nc{\lbc}{\mbox{\Large\boldmath$,$}}
\nc{\lc}{\mbox{\Large$,$}}
\nc{\Lall}{\mbox{\Large$\forall\;$}}
\nc{\bc}{\mbox{\boldmath$,$}}
\nc{\ra}{\rightarrow}
\nc{\ci}{\circ}
\nc{\cc}{\!\ci\!}
\nc{\lra}{\longrightarrow}
\nc{\imp}{\Rightarrow}
\rnc{\iff}{\Leftrightarrow}
\nc{\inc}{\mbox{$\,\subseteq\;$}}
\rnc{\subset}{\inc}
\def\sw#1{{\sb{(#1)}}}
\nc{\0}{\sb{(0)}}
\newcommand{\Boxneu}{\square}
\def\tr{{\rm tr}}
\def\Tr{{\rm Tr}}
\def\st{\stackrel}
\def\<{\langle}
\def\>{\rangle}
\def\d{\mbox{$\mathop{\mbox{\rm d}}$}}
\def\id{\mathrm{id}}
\def\ker{\mbox{$\mathop{\mbox{\rm Ker$\,$}}$}}
\def\coker{\mbox{$\mathop{\mbox{\rm Coker$\,$}}$}}
\def\hom{\mbox{$\mathop{\mbox{\rm Hom}}$}}
\def\im{\mbox{$\mathop{\mbox{\rm Im}}$}}
\def\map{\mbox{$\mathop{\mbox{\rm Map}}$}}
\def\Cl{\mbox{$\mathop{\mbox{\rm Cl}}$}}
\def\o{\sp{[1]}}
\def\t{\sp{[2]}}
\def\mo{\sp{[-1]}}
\def\z{\sp{[0]}}
\def\lhom#1#2#3{{{}\sb{#1}{\rm Hom}(#2,#3)}}
\def\rhom#1#2#3{{{\rm Hom}\sb{#1}(#2,#3)}}
\def\lend#1#2{{{}\sb{#1}{\rm End}(#2)}}
\def\rend#1#2{{{\rm End}\sb{#1}(#2)}}
\def\Rhom#1#2#3{{{\rm Hom}\sp{#1}(#2,#3)}}
\def\Lhom#1#2#3{{{}\sp{#1}{\rm Hom}(#2,#3)}}
\def\Rrhom#1#2#3#4{{{\rm Hom}\sp{#1}\sb{#2}(#3,#4)}}
\def\Llhom#1#2#3#4{{{}\sp{#1}\sb{#2}{\rm Hom}(#3,#4)}}
\def\LRhom#1#2#3#4#5#6{{{}\sp{#1}\sb{#2}{\rm Hom}\sp{#3}\sb{#4}(#5,#6)}}
\def\khom#1#2{{{\rm Hom}(#1,#2)}}
\def\aten#1{\underset{#1}{\overset{\rm{alg}}{\otimes}}}
\def\mten{\otimes_{\rm{min}}}

\nc{\spp}{\mbox{${\cal S}{\cal P}(P)$}}
\nc{\cP}{\mathcal{P}}
\nc{\ob}{\mbox{$\Omega\sp{1}\! (\! B)$}}
\nc{\op}{\mbox{$\Omega\sp{1}\! (\! P)$}}
\nc{\oa}{\mbox{$\Omega\sp{1}\! (\! A)$}}
\nc{\dr}{\mbox{$\Delta_{R}$}}
\nc{\dsr}{\mbox{$\Delta_{\Omega^1P}$}}
\nc{\ad}{\mbox{$\mathop{\mbox{\rm Ad}}_R$}}
\nc{\as}{\mbox{$A(S^3\sb s)$}}
\nc{\bs}{\mbox{$A(S^2\sb s)$}}
\nc{\slc}{\mbox{$A(SL(2,\C))$}}
\nc{\suq}{\mbox{$\cO(SU_q(2))$}}
\nc{\tc}{\widetilde{can}}
\def\slq{\mbox{$\cO(SL_q(2))$}}
\def\asq{\mbox{$\cO(S_{q,s}^2)$}}
\def\esl{{\mbox{$E\sb{\frak s\frak l (2,{\Bbb C})}$}}}
\def\esu{{\mbox{$E\sb{\frak s\frak u(2)}$}}}
\def\ox{{\mbox{$\Omega\sp 1\sb{\frak M}X$}}}
\def\oxh{{\mbox{$\Omega\sp 1\sb{\frak M-hor}X$}}}
\def\oxs{{\mbox{$\Omega\sp 1\sb{\frak M-shor}X$}}}
\def\Fr{\mbox{Fr}}
\nc{\p}{{\rm pr}}
\newcommand{\can}{\operatorname{\it can}}
\newcommand{\Can}{\operatorname{\it Can}}
\def\H{\mathcal{H}}
\def\B{\mathcal{B}}
\def\P{\mathcal{P}}

\rnc{\phi}{\varphi}
\nc{\ha}{\mbox{$\alpha$}}
\nc{\hb}{\mbox{$\beta$}}
\nc{\hg}{\mbox{$\gamma$}}
\nc{\hd}{\mbox{$\delta$}}
\nc{\he}{\mbox{$\varepsilon$}}
\nc{\hz}{\mbox{$\zeta$}}
\nc{\hs}{\mbox{$\sigma$}}
\nc{\hk}{\mbox{$\kappa$}}
\nc{\hm}{\mbox{$\mu$}}
\nc{\hn}{\mbox{$\nu$}}
\nc{\hl}{\mbox{$\lambda$}}
\nc{\hG}{\mbox{$\Gamma$}}
\nc{\hD}{\mbox{$\Delta$}}
\nc{\hT}{\mbox{$\Theta$}}
\nc{\ho}{\mbox{$\omega$}}
\nc{\hO}{\mbox{$\Omega$}}
\nc{\hp}{\mbox{$\pi$}}
\nc{\hP}{\mbox{$\Pi$}}

\nc{\qpb}{quantum principal bundle}
\def\gal{-Galois extension}
\def\hge{Hopf-Galois extension}
\def\ses{short exact sequence}
\def\csa{$C^*$-algebra}
\def\ncg{noncommutative geometry}
\def\wrt{with respect to}
\def\Ha{Hopf algebra}

\def\C{{\Bbb C}}
\def\N{{\Bbb N}}
\def\R{{\Bbb R}}
\def\Z{{\Bbb Z}}
\def\T{{\Bbb T}}
\def\Q{{\Bbb Q}}
\def\cO{{\mathcal O}}
\def\cT{{\cal T}}
\def\cA{{\cal A}}
\def\cD{{\cal D}}
\def\cB{{\cal B}}
\def\cK{{\cal K}}
\def\cH{{\cal H}}
\def\cM{{\cal M}}
\def\cJ{{\cal J}}
\def\fB{{\frak B}}
\def\pr{{\rm pr}}
\def\ta{\tilde a}
\def\tb{\tilde b}
\def\td{\tilde d}
\def\rd{\mathrm{d}}
\author{Paul F. Baum}
\address{Mathematics Department, McAllister Building,  
The Pennsylvania State University,
University Park, PA  16802, USA\\
Instytut Matematyczny, Polska Akademia Nauk, ul.~\'Sniadeckich 8, Warszawa, 00-656 Poland}
\email{baum@math.psu.edu}
\author{Kenny De Commer}
\address{Department of Mathematics,
Vrije Universiteit Brussel,
Pleinlaan 2,
1050 Brussels, Belgium}
\email{kenny.de.commer@vub.ac.be}
\author{Piotr~M.~Hajac}
\address{Instytut Matematyczny, Polska Akademia Nauk, ul.~\'Sniadeckich 8, Warszawa, 00-656 Poland}
\email{pmh@impan.pl}
\title[Free actions of quantum groups]{\large Free actions of compact quantum groups\\ 
\vspace*{2mm} on unital C*-algebras\\ ~}
\vspace*{-15mm}\maketitle
\vspace*{-5mm}\begin{abstract}\normalsize
Let $F$ be a field, $\Gamma$ a finite group, and $\map(\Gamma,F)$ the Hopf algebra of all set-theoretic maps $\Gamma\rightarrow F$.
If $E$ is a finite field extension of $F$ and $\Gamma$ is its Galois group, the extension is Galois if and
only if the canonical map $E\otimes_FE\rightarrow E\otimes_F\map(\Gamma,F)$ resulting from viewing $E$ as a $\map(\Gamma,F)$-comodule
is an isomorphism. Similarly, a finite covering space is regular if and only if the analogous canonical map is an isomorphism. In this paper
we extend this point of view to actions of compact quantum groups on unital $C^*$-algebras. We prove that such an action is $C^*$-free
if and only if the canonical map (obtained using the underlying Hopf algebra of the compact quantum group) is an isomorphism.
In particular, we are able to express the freeness of a compact Hausdorff topological group action on a compact Hausdorff topological space in 
algebraic terms. As an application, we  show that a field of $C^*$-free actions yields a global $C^*$-free action.
\end{abstract}
{\small
\tableofcontents
\setcounter{tocdepth}{2}
}
\newpage

\section*{Introduction}
\setcounter{equation}{0}

\noindent A \emph{compact quantum group} \cite{w-sl87,w-sl98} is a unital $C^*$-algebra $H$ with a given unital injective $*$-homorphism $\Delta$ 
(referred to as
comultiplication)
\[
\Delta \colon H\longrightarrow H\underset{\text{min}}{\otimes}H
\]
that is coassociative, i.e.\ it renders the diagram 
\[
\xymatrix{
H\ar[r]^{\Delta}\ar[d]^{\Delta}& H\!\!\overset{\phantom{min}}{\underset{\mathrm{min}}{\otimes}}\!\! H\ar[d]^{\Delta\otimes\id}\\
H\!\!\overset{\phantom{min}}{\underset{\mathrm{min}}{\otimes} }\!\!H \ar[r]_{\id\otimes\Delta}&H\!\!\overset{\phantom{min}}
{\underset{\mathrm{min}}{\otimes} }\!\!H\!\!\underset{\mathrm{min}} {\otimes} \!\!H
}
\]
commutative, and such that the two-sided cancellation property holds:
\[
\{(a\otimes 1)\Delta(b)\;|\;a,b\in H\}^{\mathrm{cls}}
=H\underset{\mathrm{min}}{\otimes}H=
\{\Delta(a)(1\otimes b)\;|\;a,b\in H\}^{\mathrm{cls}}.
\]
Here $\otimes_\mathrm{min}$ denotes the spatial tensor product of $C^*$-algebras and 
$\mathrm{cls}$ denotes the closed linear span of a subset of a Banach space. 

Let $A$ be a unital $C^*$-algebra and
$\delta:A\rightarrow A\otimes_{\mathrm{min}}H$ an
injective  unital $*$-homomorphism. We call $\delta$ a {\em coaction}  of $H$ on $A$
 (or an action of the compact quantum group $(H,\Delta)$ on $A$) iff
\vspace*{-2.5mm}
\begin{enumerate}
\item
$(\delta\otimes\mathrm{id})\circ\delta=(\mathrm{id}\otimes\Delta)\circ\delta$
(coassociativity),
\item
$\{\delta(a)(1\otimes h)\;|\;a\in A,\,h\in H\}^{\mathrm{cls}}=
A\underset{\mathrm{min}}{\otimes}H$ (counitality).
\end{enumerate}

We shall consider three properties of coactions. 
\begin{definition}[\cite{e-da00}]
The coaction $\delta:A\rightarrow A\otimes_{\mathrm{min}}H$ is \emph{$C^*$-free} iff
\[
\{(x\otimes 1)\delta(y)\;|\;x,y\in A\}^{\mathrm{cls}}=
A\underset{\mathrm{min}}{\otimes}H.
\]
\end{definition}

Given a compact quantum group $(H,\Delta)$, we denote by
 $\cO (H)$ its dense
Hopf $*$-subalgebra spanned by the matrix coefficients of
irreducible unitary corepresentations~\cite{w-sl98,mv98}.
This is Woronowicz's Peter-Weyl theory in the case of compact quantum groups.
Moreover, denoting by $\otimes$  the purely algebraic tensor product over the field
	$\mathbb{C}$ of complex numbers, we define the
\emph{Peter-Weyl subalgebra} of $A$ (cf.~\cite{p-p95,s-pm11}) as
\[
\cP_H(A):=\{\,a\in A\,| \,\delta(a)\in A\otimes\cO (H)\,\}.
\]
Using the coassociativity of $\delta$, one can  check that $\cP_H(A)$ is a right $\cO (H)$-comodule algebra. In particular, $\P_H(H)=\cO(H)$.
The assignment $A\mapsto\cP_H(A)$ is functorial with respect to
equivariant unital $*$-homomorphisms and comodule algebra maps. We call it the
\emph{Peter-Weyl functor}.
\begin{definition}
The coaction $\delta:A\rightarrow A\otimes_{\mathrm{min}}H$ satisfies the \emph{Peter-Weyl-Galois (PWG) condition}
iff the canonical map 
\begin{align}\label{pwg}
can\colon \mathcal{P}_H(A)\underset{B}{\otimes} \mathcal{P}_H(A)&\longrightarrow 
\mathcal{P}_H(A)\otimes \mathcal{O}(H)\nonumber\\
can\colon x\otimes y&\longmapsto (x\otimes 1)\delta(y)
\end{align}
is bijective. Here $B=A^{\mathrm{co}H}:=\{a\in A\;|\;\delta(a)=a\otimes 1\}$ is the unital
$C^*$-subalgebra of coaction-invariants.
\end{definition}
\vspace*{-4mm}
\noindent
Throughout this paper  the tensor product over an algebra denotes the purely algebraic tensor product over 
that algebra.

\begin{definition}
The coaction $\delta:A\longrightarrow A\otimes_{\mathrm{min}}H$ is \emph{strongly monoidal} iff
for all left\linebreak
  \mbox{$\mathcal{O}(H)$-co}modules $V$ and $W$ the map
\begin{align*}
\beta: (\mathcal{P}_H(A)\Box V)   \underset{B}{\otimes} (\mathcal{P}_H(A)\Box W) &
\longrightarrow \mathcal{P}_H(A)\Box(V\otimes W)\\ 
\Big(\sum_ia_i\otimes v_i\Big)\otimes \Big(\sum_jb_j\otimes w_j\Big) &
\longmapsto \sum_{i,j}a_ib_j\otimes (v_i\otimes w_j)
\end{align*}
is bijective. 
\end{definition}
\vspace*{-2mm}
\noindent
In the above definition, we have used the cotensor product 
\[\label{pwcot}
\mathcal{P}_H(A)\Box V:=\{t\in \mathcal{P}_H(A)\otimes V\;|\;(\delta\otimes\id)(t)=
(\id\otimes{}_V\Delta)(t)\},
\]
where ${}_V\Delta\colon V\to \mathcal{O}(H)\otimes V$ is the given left coaction of 
$\mathcal{O}(H)$ on $V$. 
The coaction of 
$\mathcal{O}(H)$ on $V\otimes W$ is the diagonal coaction. 

The theorem of this paper is:
\begin{theorem}\label{theoGal}
Let $A$ be a unital $C^*$-algebra equipped with an action of a compact quantum group $(H,\Delta)$ given by
$\delta\colon A\rightarrow A\otimes_{\mathrm{min}}H$.  Then the following are equivalent:
\vspace*{-3mm}\begin{enumerate}
\item
The action of $(H,\Delta)$ on $A$ is $C^*$-free. 
\item
The action of $(H,\Delta)$ on $A$ satisfies the Peter-Weyl-Galois condition.
\item
The action of $(H,\Delta)$ on $A$ is strongly monoidal.
\end{enumerate}
\end{theorem}\vspace*{-5mm}
\noindent
Note that of the three equivalent conditions, the first uses functional analysis, the second is algebraic, and the 
third is categorical. The difficult implication, which is the core of the theorem, is (1)$\;\Longrightarrow\;$(2).
It proves that, for any $C^*$-free action, there exists a strong connection, a key technical device for
index-pairing computations (e.g.~\cite{hms03}).
In the spirit of Woronowicz's Peter-Weyl theory, our result states that the original functional analysis 
formulation of free action is equivalent to
 the much more algebraic PWG-condition.  

We now  proceed to explain  our main result in the classical setting.
Let $G$ be a compact Hausdorff topological group acting (by a continuous right action) on a compact 
Hausdorff topological space $X$
\[
X\times G\longrightarrow X.
\]
 It is immediate that the action is free i.e.\ $xg=x \Longrightarrow g=e$ 
(where $e$ is the identity element of~$G$) if and only if 
 \begin{align}
 X\times G &\longrightarrow X\underset{X/G}{\times}X\nonumber\\
 (x,g) &\longmapsto (x,xg)
 \end{align}
 is a homeomorphism. Here $X\times_{X/G}X$ is the subset of $X\times X$ consisting of pairs $(x_1, x_2)$ 
such that $x_1$ and $x_2$
 are in the same
 $G$-orbit.

 This is equivalent to the assertion that the $*$-homomorphism
 \[\label{ellwood}
 C(X\underset{X/G}{\times}X)\longrightarrow C(X\times G)
 \]
 obtained from the above map $(x,g)\mapsto (x,xg)$ is an isomorphism.
 Here, as usual, $C(Y)$ denotes the commutative $C^*$-algebra of all continuous complex-valued functions on the compact Hausdorff space $Y$.

 In turn, the assertion that $*$-homomorphism \eqref{ellwood} is an isomorphism is readily proved equivalent to
 \[\label{classcls}
\{(x\otimes 1)\delta(y)\;|\;x,y\in C(X)\}^{\mathrm{cls}}=
C(X)\underset{\mathrm{min}}{\otimes}C(G),
 \]
 where
\[\label{dualco}
\delta\colon C(X)\longrightarrow C(X)\underset{\mathrm{min}}{\otimes}C(G),\quad (\delta(f)(g))(x)=f(xg),
\]
is the $*$-homomorphism obtained from the action map $X\times G\rightarrow X$. 
Hence, in the case of a compact group acting on a compact space, freeness
agrees with $C^*$-freeness as defined in the setting of a compact quantum group acting on a unital $C^*$-algebra. Thus 
Theorem~\ref{theoGal} provides the following characterization of free actions in the classical case.
\bth\label{classical}
Let $G$ be a compact Hausdorff group acting continuously on a compact Hausdorff
space~$X$. Then the action is free if and only if the canonical map  
\[\label{canclas}
can\colon  \mathcal{P}_{C(G)}(C(X))\underset{C(X/G)}{\otimes} \mathcal{P}_{C(G)}(C(X))\longrightarrow \mathcal{P}_{C(G)}(C(X))\otimes \mathcal{O}(C(G))
\]
is an isomorphism.
\ethe

Observe that even in the above special case of a compact group acting on a compact space, 
a proof is required for the equivalence of ``free action" and the bijectivity of the canonical map
(PWG-condition). Theorem~\ref{classical} brings a new algebraic tool (strong connection)    
to the realm of compact principal bundles.

In this classical setting, 
 the Peter-Weyl algebra $\mathcal{P}_{C(G)}(C(X))$  is
 the algebra of continuous global sections of the associated
 bundle of algebras \mbox{$X\underset{G}{\times}\mathcal{O}(C(G))$}:
\[
\mathcal{P}_{C(G)}(C(X))=\Gamma\Big(X\underset{G}{\times}\mathcal{O}(C(G))\Big).
\]
Here $\mathcal{O}(C(G))$ is the subalgebra of $C(G)$ generated by matrix coefficients of irreducible 
representations of~$G$. 
We view 
$\mathcal{O}(C(G))$ as a representation space of $G$ via the formula 
\[\label{015}
\big(\varrho(g)(f)\big)(h):=f(g^{-1}h). 
\]
The algebra $\mathcal{O}(C(G))$ is  topologized
as the direct limit of its finite dimensional subspaces. Multiplication and addition of sections is pointwise.

Note that, since $\mathcal{O}(C(G))$ is cosemisimple, it belongs to the category of representations of $G$ 
that are purely algebraic direct sums of finite-dimensional representations of $G$. We denote this category by 
$\mathrm{FRep}^\oplus(G)$. Due to the cosemisimplicity of  $\mathcal{O}(C(G))$, the following formula for 
the left coaction of $\mathcal{O}(C(G))$ on $V$ 
\[\label{repcorep} 
({}_V\Delta(v))(g)=\varrho(g^{-1})(v),\quad\text{where }\;\varrho:\;G\longrightarrow GL(V)\;
\text{is a  representation},
\] 
establishes an equivalence of $\mathrm{FRep}^\oplus(G)$ with the category of \emph{all} left 
$\mathcal{O}(C(G))$-comodules. As with the special case $V=\mathcal{O}(C(G))$, 
all vector spaces in this category are topologized  as the direct limits of their finite dimensional subspaces.

 Theorem~\ref{classical} unifies free actions of compact Hausdorff
groups on compact Hausdorff spaces and principal actions of affine
algebraic groups on affine schemes~\cite{dg70,s-p04}. Thus the main result of our paper 
might be viewed as continuing the Atiyah-Hirzebruch program of transferring ideas (e.g.\ K-theory)
from algebraic geometry to topology~\cite{ah59,ah61}. In the same spirit, our main theorem 
(Theorem~\ref{theoGal}) unifies the $C^*$-algebraic concept of free actions of compact quantum groups
\cite{e-da00} with the Hopf-algebraic concept of principal coactions~\cite{hkmz11}.
Theorem~\ref{theoGal} implies the existence of  strong connections \cite{h-pm96} for free actions
of compact quantum groups on unital $C^*$-algebras (connections on compact quantum principal bundles)
thus providing a theoretical foundation for the plethora of concrete constructions studied over 
the past two decades within the general framework noncommutative geometry~\cite{c-a94}.
In this paper, we apply Theorem~\ref{theoGal} to fields of $C^*$-algebras (Corollary~\ref{application}).

The paper is organized as follows. In Section~1, we prove the key part of our main theorem, 
that is the equivalence of $C^*$-freeness
and the Peter-Weyl-Galois condition. In Section~2, 
we consider
the general algebraic setting of principal coactions. 
Following Ulbrich \cite{u-kh89} and Schauenburg \cite{s-p04},
we prove that the principality of a comodule algebra $\mathcal{P}$ over
a Hopf algebra $\mathcal{H}$
is equivalent to the exactness and strong monoidality of the cotensor product functor 
$\mathcal{P}\Box_\mathcal{H}$.
In particular, this proves the 
equivalence of the Peter-Weyl-Galois condition and strong monoidality for actions of compact quantum groups, 
thus completing the proof
of the main theorem.

 Although Theorem~\ref{classical} is a special
case of Theorem~\ref{theoGal}, the proof we give of  Theorem~\ref{classical}  is not a special case of the 
proof of Theorem~\ref{theoGal}. Therefore we treat 
Theorem~\ref{classical} separately, and prove it in Section~3.
The proof  uses the strong monoidality (i.e.\ the preservation of tensor products) 
of the Serre-Swan equivalence
and a general algebraic argument (Corollary~\ref{4cor}) of Section~2. In Section~4, we give a vector-bundle
 interpretation of the aforementioned general algebraic argument. This provides a much desired
translation between
the algebraic and topological settings.

In Section~5, as an application of our main result, we prove that if a  unital $C^*$-algebra $A$ equipped with 
an action of a compact quantum group can be fibred over
a compact Hausdorff space $X$ with the PWG-condition valid on each fibre, then the PWG-condition is valid for 
the action on $A$.
We end with an appendix discussing the well-known fact  that regularity of a finite covering is equivalent to 
bijectivity of the canonical map~\eqref{canclas}.

\section{Equivalence of $C^*$-freeness and Peter-Weyl-Galois}\label{pfmain}
\setcounter{equation}{0}

The implication ``PWG-condition\hspace{2mm}$\Longrightarrow$\hspace{2mm}$C^*$-freeness''
is proved as follows. The PWG-condition immediately implies that
\[
[(\P_H(A)\otimes \mathbb{C})\delta(\P_H(A))]^{\mathrm{ls}} = \P_H(A)\otimes \cO (H).
\]
As the right-hand 
side is a dense subspace of $A\underset{\rm min}{\otimes} H$ (see \cite[Theorem~1.5.1]{p-p95} 
and \cite[Proposition 2.2]{s-pm11}), we obtain the density condition defining $C^*$-freeness.

For the converse implication 
``PWG-condition\hspace{2mm}$\Longleftarrow$\hspace{2mm}$C^*$-freeness'' we
need some preparations. If $(V,\delta_V)$ is a finite-dimensional right 
 $H$-comodule,
we write $H_V$ for the smallest vector subspace of $H$ such that
$\delta_V(V)\subseteq V\otimes H_V$. We write 
\begin{equation}\label{eqCom} A_V := \{a \in A\mid \delta(a)\in A\otimes H_V\}
\end{equation} Note that in the case $(A,\delta)=(H,\Delta)$, we have $A_V =
H_V$. Thus $H_V$ is a coalgebra.

One can define a continuous projection map $E_V$ from $A$ onto $A_V$ as follows
\cite[Theorem~1.5.1]{p-p95}. Let us call two finite-dimensional comodules of $H$
\emph{disjoint} if the set of morphisms between them only contains the zero map.
Then $E_V$ is the unique endomorphism of $A$ which is the identity on $A_V$ and
which vanishes on $A_W$ for $W$ any finite-dimensional comodule disjoint from
$V$. In the special case of $(A,\delta)=(H,\Delta)$, we use the notation
$e_V$ instead of $E_V$. The equivariance property
\[\label{equivariance}
\delta\circ E_V = (\id\otimes e_V)\circ \delta.
\]
is proved by a straightforward verification. When $V$ is the trivial
representation, we write $E_V=E_B$ and $e_V=\varphi_H$,
where $B=A^{\mathrm{co}H}$ is the algebra of coaction invariants and
$\varphi_H$ is the invariant
state on $H$. Then the formula \eqref{equivariance} specializes to
\begin{equation}\label{eqDefEB}
 E_B =  (\id\otimes \varphi_H)\circ\delta.
\end{equation}

The key lemma in the proof of Theorem~\ref{theoGal} is:
\begin{lemma}[Theorem 1.2 in \cite{dy12}]\label{theoMak} 
 Let $\delta\colon A\to A\otimes_{\mathrm{min}}H$ be a $C^*$-free coaction, 
and let $V$ be a finite-dimensional $H$-comodule. 
Then $A_V$ is finitely generated projective as a right $B$-module.
\end{lemma}
\vspace*{-4mm} \noindent 
Note that in the classical case $X\times G\to X$, we have $H=C(G)$ and $B=C(X/G)$. 
The $B$-module 
$A_V$ is then $\Gamma(X\times_GH_V)$, and thus it is finitely generated projective. 

Define a $B$-valued inner product on $A_V$ by
\[\label{B-inner}
\langle a,b\rangle_B := E_B(a^*b).    
\]
\begin{lemma}[Corollary 2.6 in \cite{dy12}]\label{theoMak2} 
The $B$-valued inner product \eqref{B-inner} makes $A_V$ a (right) Hilbert $B$-module \cite{l-ec95}. 
The Hilbert module norm $\|a\|_B := \|\langle a,a\rangle_B\|^{1/2}$ is equivalent to the $C^*$-norm of $A$ 
restricted to $A_V$.
\end{lemma}

We will need the following lemma concerning the interior tensor product of Hilbert modules.
\begin{lemma}[cf.\ Proposition 4.5 in \cite{l-ec95}]\label{lemtens} 
Let $C$ and $D$ be unital $C^*$-algebras, and let 
$(\mathscr{E},\langle\,\cdot\,,\,\cdot\,\rangle_C)$ be a right Hilbert $C$-module that is finitely generated 
projective as a right $C$-module. 
Let $(\mathscr{F},\langle\,\cdot\,,\,\cdot\,\rangle_D)$ be an arbitrary right Hilbert $D$-module, and 
$\pi: C\rightarrow \mathcal{L}(\mathscr{F})$ be a unital \mbox{$^*$-homomorphism} 
of $C$ into the $C^*$-algebra of adjointable operators on~$\mathscr{F}$. 
Then the algebraic tensor product
 $\mathscr{E}\otimes^{\mathrm{alg}}_{C} \mathscr{F}$ is a right Hilbert $D$-module with respect to the 
inner product  given by
\[
\langle x\otimes y, z\otimes w\rangle := \langle y,\pi(\langle x,z\rangle_C)w \rangle_D.
\]  
\end{lemma}
\vspace*{-5mm}
\begin{proof} We need to prove that the semi-norm $\|z\| = \|\langle z,z\rangle_D\|^{1/2}$ on 
$\mathscr{E}\otimes^{\mathrm{alg}}_{C} \mathscr{F}$ is in fact a norm with respect to which 
$\mathscr{E}\otimes^{\mathrm{alg}}_{C} \mathscr{F}$ is complete. 
The statement obviously holds for $\mathscr{E} = C^n$, the $n$-fold direct sum of the standard right $C$-
module $C$. 
Since $\mathscr{E}$ is finitely generated 
projective, $\mathscr{E}$ can be realized as a direct summand of $C^n$, so that the conclusion also applies 
for this case.
\end{proof}  
\vspace*{-2mm}
We are now ready to prove the implication
``PWG-condition\hspace{2mm}$\Longleftarrow$\hspace{2mm}$C^*$-freeness''.
By the \mbox{$C^*$-freeness} assumption, the image of $\can$ is dense in $A\otimes H$. In
particular, for a given finite-dimensional comodule $V$ and any $h\in H_V$, we
can find a sequence $k_n \in \N$ and elements $p_{n,i}$ and $q_{n,i}$ in
$\P_H(A)$ with $1\leq i\leq k_n$ such that 
\begin{equation}\label{eqLim}
\sum_{i=1}^{k_n} (p_{n,i}\otimes 1)\delta(q_{n,i})
 \underset{n\rightarrow \infty}{\longrightarrow} 1\otimes h
\end{equation} 
in the $C^*$-norm. Applying $\id\otimes e_V$ to this expression, and using
\eqref{equivariance}, we see that we can take $q_{n,i} \in A_V$.

Applying $\delta$ to the first leg of \eqref{eqLim} and using coassociativity,
we obtain 
\begin{equation}\label{eqLim2}
\sum_{i=1}^{k_n} (\delta(p_{n,i})\otimes
1)(\id\otimes \Delta)(\delta(q_{n,i})) \underset {n\rightarrow
\infty}{\longrightarrow} 1\otimes 1\otimes h.
\end{equation}

Observe now that, since $q_{n,i}\in A_V$, by \eqref{eqCom} we obtain
$(\id\otimes \Delta)(\delta(q_{n,i}))\in A_V\otimes H_V\otimes H_V$. Hence the
left-hand side of \eqref{eqLim2} belongs to the tensor product
$(A\otimes_{\mathrm{min}} H)\otimes H_V$. As $H_V$ is finite dimensional, the
restriction of the antipode $S$ of $\cO (H)$ to $H_V$ is continuous. Therefore,
we can apply $S$ to the third leg of \eqref{eqLim2} to
conclude
\[
\sum_{i=1}^{k_n}(\delta(p_{n,i})\otimes 1)(\id\otimes (\id \otimes
S)\circ\Delta)(\delta(q_{n,i})) 
\underset{n\rightarrow \infty}{\longrightarrow}
1\otimes 1\otimes S(h).
\]

Again by the finite dimensionality of $H_V$, multiplying the second and third
legs is a continuous operation, so that \[\label{limpq}\sum_{i=1}^{k_n}
\delta(p_{n,i})(q_{n,i}\otimes 1) \underset{n\rightarrow
\infty}{\longrightarrow} 1\otimes S(h).\]
Since $S(h) \in H_{\bar{V}}$,
where $\bar{V}$ is the contragredient of $V$, applying $\id \otimes e_{\bar{V}}$
to the above limit, and using the equivariance property \eqref{equivariance},
we infer that in the above limit we can choose $p_{n,i}\in A_{\bar{V}}$.

Consider now the right $B$-module map
\[
G_V:A_{\bar{V}}\overset{\textrm{alg}}{\underset{B}{\otimes}} A_V\longrightarrow
A_{\bar{V}\otimes V}\otimes H_{\bar{V}}, \qquad a\otimes b \longmapsto
\delta(a)(b\otimes 1).
\] 
By Lemma \ref{theoMak} and Lemma~\ref{lemtens}, the left-hand side becomes an
interior tensor product of
right Hilbert $B$-modules for the inner product \[\langle c\otimes d,a\otimes
b\rangle_{B} = E_B(d^*E_B(c^*a)b).\]

 On the other hand, equipping $H_{\bar{V}}$ with the Hilbert
space structure $\langle h,k\rangle = \varphi_H(h^*k)$, the right-hand side
is a right Hilbert $B$-module by
\[\langle b\otimes h, a\otimes g\rangle_B = \varphi_H(h^*g)E_B(b^*a).\]
From these formulas and \eqref{eqDefEB}, it follows that $G_V$ is an isometry
between these Hilbert modules. Hence the range of $G_V$ is closed.

From \eqref{limpq} and the equivalence of $C^*$- and Hilbert $C^*$-module norms
in Lemma~\ref{theoMak2}, 
it follows that the range of $G_V$ contains $1\otimes S(h)$. Therefore, as the
domain of $G_V$ is an \emph{algebraic} tensor product, we can find a
finite number of elements $p_i,q_i \in \P_H(A)$
such that \[\sum_{i} \delta(p_i)(q_i\otimes 1) = 1\otimes S(h).\] Now applying
 the map $a\otimes g \mapsto (1\otimes S^{-1}(g))\delta(a)$ to both sides yields
\[
\sum_{i} (p_i\otimes 1)\delta(q_i) = 1\otimes h.
\] 
As $h$ was arbitrary in $\cO (H)$, it follows that $\can$ is surjective.

Finally, as the Hopf algebra $\cO (H)$ is cosemisimple, according to
\cite[Remark 3.9]{s-hj90},  bijectivity of the canonical map $\can$  follows
from surjectivity. 
This completes the proof of the implication 
``PWG-condition\hspace{2mm}$\Longleftarrow$\hspace{2mm}$C^*$-freeness''.

\section{Equivalence of principality and strong monoidality}
\setcounter{equation}{0}

The framework of principal comodule algebras unifies in one category many algebraically constructed non-
commutative examples and 
classical compact principal bundles.

\begin{definition}[\cite{bh04}]\label{principality}
Let $\H$ be a Hopf algebra with bijective antipode, and let\linebreak
 $\Delta_\P\colon\P\to\P\otimes\H$ be a coaction making
$\P$ an $\H$-comodule algebra. We call $\P$   {\em principal}
 if and only if:
\vspace*{-5mm}\begin{enumerate}
\item
$\P{\otimes}_\B \P\ni p \otimes q \mapsto\can(p \otimes q):=
(p\otimes 1)\Delta_\P(q)\in \P \otimes \H$
is bijective, where\\ $\B=\P^{\mathrm{co}\H}:=\{p\in \P\;|\;\Delta_\P(p)=p\otimes 1\}$;
\item
there exists a left $\B$-linear right $\H$-colinear splitting of the
multiplication map \mbox{$\B\otimes \P\to \P$}.
\end{enumerate}
\end{definition}
\vspace*{-5mm}\noindent
Here (1) is the Hopf-Galois condition and (2) is the right equivariant left
projectivity of $\P$.

Alternately, one can approach principality through strong connections:
\begin{definition}
Let $\H$ be a Hopf algebra with  bijective antipode $S$,
and  $\Delta_\P\colon\P\to\P\otimes\H$ be a coaction
making $\P$ a right $\H$-comodule algebra.
A \emph{strong connection $\ell$} on $\P$ is a unital linear map $\ell :\H \rightarrow \P \otimes \P$ satisfying:
\newpage\begin{enumerate}
\item
$(\mathrm{id}\otimes \Delta_\P) \circ
\ell = (\ell \otimes \mathrm{id}) \circ \Delta$;
\item
$({}_\P \Delta \otimes \mathrm{id}) \circ
\ell = (\mathrm{id} \otimes \ell) \circ
\Delta$, where
${}_\P \Delta:=(S^{-1}\otimes\id)\circ\mathrm{flip}\circ\Delta_\P$;
\item
$\widetilde{can} \circ \ell=1 \otimes \mathrm{id}$, where
$\widetilde{can}\colon \P\otimes \P\ni p\otimes q\mapsto (p\otimes 1)\Delta_\P(q)\in \P\otimes\H$.
\end{enumerate}
\end{definition}
\vspace*{-5mm}\noindent
One can  prove (see \cite{bh} and references therein) that
a comodule algebra is principal if and only if it admits a strong
connection.

If $\Delta_M\colon M\to M\otimes C$ is a  coaction making $M$ a right comodule over a coalgebra $C$ and $N$ is a left $C$-
comodule via a  coaction ${}_N\Delta\colon N\to C\otimes N$,
then we define their {\em cotensor product} as
\[
M\underset{C}{\Box}N:=\{t\in M\otimes N\;|\;(\Delta_M\otimes\id)(t)=
(\id\otimes{}_N\Delta)(t)\}.
\]
In particular, for a right $\H$-comodule algebra $\P$ and a left
$\H$-comodule $V$, we observe that $\P\Box_\H V$ is a left $\P^{\co\H}$-
module in a natural way.
One of the key properties of principal comodule algebras is that,
for any finite-dimensional left $\H$-comodule $V$, the left $\P^{\co \H}$-module
$\P\Box_\H V$ is finitely generated projective~\cite{bh04}. Here $\P$ plays the
role of a principal bundle and $\P\Box_\H V$ plays the role of an associated
vector bundle. Therefore, we call $\P\Box_\H V$ an \emph{associated module}.

Principality can also be characterized by the exactness and strong monoidality of the cotensor functor. This 
characterisation uses the notion of coflatness of a comodule: a right comodule is \emph{coflat} if and only if 
cotensoring it with left comodules preserves exact sequences.
\begin{theorem}\label{4thm}
Let $\H$ be a Hopf algebra with bijective antipode, and $\P$ a right $\H$-comodule
algebra. Then $\P$ is principal if and only if $\P$ is right $\H$-coflat and for all
left  $\H$-comodules $V$ and $W$ the map
\begin{align*}
\beta: (\P\Box V)   \underset{\B}{\otimes} (\P\Box W) &\longrightarrow \P\Box(V\otimes W)\\ 
\Big(\sum_ia_i\otimes v_i\Big)\otimes \Big(\sum_jb_j\otimes w_j\Big) 
&\longmapsto \sum_{i,j}a_ib_j\otimes (v_i\otimes w_j)
\end{align*}
is bijective. 
In other words, $\P$ is principal if and only if 
the cotensor product functor is exact and strongly monoidal with respect to the above map $\beta$.
\end{theorem}
\begin{proof}
\vspace*{-2mm}
The proof relies on putting together \cite[Theorem I]{s-hj90}, \cite[Theorem~6.15]{s-p98},
\cite[Theorem~2.5]{bh04} and \cite[Theorem~5.6]{ss05}. First assume that $\P$ is principal. Then $\P$ is 
right equivariantly projective, and it follows from \cite[Theorem~2.5]{bh04}
that $\P$ is faithfully flat. Now we can apply \cite[Theorem~6.15]{s-p98} to conclude that $\beta$ is bijective. 
Furthermore, 
by \cite[Theorem I]{s-hj90}, the faithful flatness of $\P$
implies the coflatness of $\P$.  Conversely, assume that cotensoring with $\P$ is exact and strongly monoidal 
with respect to $\beta$. Then substituting $\H$ for $V$ and $W$ yields the Hopf-Galois condition. Now 
\cite[Theorem~5.6]{ss05} implies the
equivariant projectivity of $\P$.
\end{proof}
\vspace*{0mm}

\begin{corollary}\label{4cor}
Let $A$ be a unital $C^*$-algebra equipped with an action of a compact quantum group $(H,\Delta)$ given by
$\delta\colon A\rightarrow A\otimes_{\mathrm{min}}H$.  Then the following are equivalent:
\vspace*{-3mm}\begin{itemize}
\item
The action of $(H,\Delta)$ on $A$ satisfies the Peter-Weyl-Galois condition.
\item
The action of $(H,\Delta)$ on $A$ is strongly monoidal.
\end{itemize}
\end{corollary}
\begin{proof}
The Hopf algebra $\cO(H)$ always has bijective antipode. It follows from \cite[Theorem~4.2]{w-sl87}
and \cite[Lemma~2.4]{bh04} that any comodule over this Hopf algebra is coflat. Hence 
\cite[Theorem~5.6]{ss05} implies  that
the equivariant projectivity condition (i.e.\ Condition~(2) of Definition~\ref{principality}) is valid for any
$\cO(H)$-comodule algebra such that the canonical map is bijective. 
The corollary now follows from Theorem~\ref{4thm}.
(As an alternative to \cite[Theorem~5.6]{ss05}, one can use the combination of 
 \cite[Theorem~4]{bb08} and \cite[Lemma~2.2]{bh04}.)
\end{proof}

\section{The classical case}
\setcounter{equation}{0}

In this section, we prove our main result in the classical case, i.e.\ we prove Theorem~\ref{classical}.
As in the proof of the general noncommutative case, we rely on the fact that the module of continuous
sections of an associated vector bundle is finitely generated projective. However, unlike in
 the proof in Section~\ref{pfmain}, herein we first prove strong monoidality, and then conclude
the PWG-condition. An entirely different proof of Theorem~\ref{classical}, using local triviality,
can be found in~\cite{bh14}.

To be consistent with general notation, 
we should only use  $C^*$-algebras $C(G)$, $C(X)$,
etc., rather than spaces themselves. However, this would make formulas too cluttered, so that
throughout this section we consistently omit writing~$C(\ )$ in the subscript and the argument
of the Peter-Weyl functor.

The implication ``PWG-condition\hspace{2mm}$\Longrightarrow$\hspace{2mm}freeness''
is 
proved as follows. The PWG-condition immediately implies that
\[
(\P_G(X)\otimes \mathbb{C})\delta(\P_G(X)) = \P_G(X)\otimes \cO (G).
\]
As the right-hand side is a dense subspace of $C(X)\otimes_{\rm min} C(G)$, 
we obtain the density condition~\eqref{classcls}. The latter is equivalent to  freeness, 
as explained in the introduction.

For the converse implication ``PWG-condition\hspace{2mm}$\Longleftarrow$\hspace{2mm}freeness'' we shall
use the Serre-Swan theorem.

\begin{theorem}[\cite{s-r62}]
Let $Y$ be a compact Hausdorff topological space. Then a $C(Y)$-module is finitely generated and projective if and only if it is isomorphic to the module of 
continuous global sections of a vector bundle over~$Y$.
\end{theorem}


For  a compact Hausdorff topological space  $Y$,  we denote by $\mathrm{Vect}(Y)$
the category of $\mathbb{C}$ vector bundles on $Y$.
 An object in
 $\mathrm{Vect}(Y)$ is a $\mathbb{C}$ vector bundle $E$ with base space ~$Y$.
 The projection of $E$ onto $Y$ is denoted by
 $\pi \colon E\rightarrow Y$.
A {section} of $E$ is a continuous map
\[
s \colon Y\longrightarrow E \quad\quad\text{with } \pi\circ s = \id_Y\,.
\]
A morphism in $\mathrm{Vect}(Y)$ is a vector bundle map
\[
\varphi \colon E\longrightarrow F,\qquad\varphi\colon E_y\longrightarrow F_y,\qquad\forall y\in Y.
\]
Note that $E_y$ and $F_y$ are both finite-dimensional vector spaces over $\mathbb{C}$ and 
$\varphi\colon E_y\rightarrow F_y $ 
is a linear transformation in the sense of standard linear algebra.

View the commutative $C^*$-algebra $C(Y)$ as a commutative ring with unit. Denote by
$\mathrm{FProj}(C(Y))$ the category of finitely generated projective $C(Y)$-modules.
An object in   the category 
$\mathrm{FProj}(C(Y))$ is a finitely generated projective $C(Y)$-module.
A morphism in $\mathrm{FProj}(C(Y))$ is a map of $C(Y)$-modules
$
\psi \colon M\rightarrow N
$.

If $E$ is a $\mathbb{C}$ vector bundle on $Y$, then
$\Gamma(E)$ denotes the $C(Y)$-module consisting of all continuous sections of $E$.
The module structure is pointwise, i.e.\ for $s_1, s_2, s \in \Gamma(E)$, \mbox{$f\in C(Y)$}, $y\in Y$,
\[
(s_1 + s_2)(y) = s_1(y) + s_2(y),
\qquad
(fs)(y) = f(y)s(y).
\]
According to the Serre-Swan theorem, the functor $\Gamma$
\[
\mathrm{Vect}(Y)\longrightarrow\mathrm{FProj}(C(Y)),\qquad E \longmapsto \Gamma(E),
\]
 is an equivalence of categories and preserves all the basic properties of the two categories.
In particular, $E \longmapsto \Gamma(E)$ preserves $\oplus$ and $\otimes$:
\begin{align}
\Gamma(E\oplus F) &= \Gamma(E)\oplus\Gamma(F),
\\
\Gamma(E\otimes F) &=\Gamma(E)\underset{C(Y)}{\otimes}\Gamma(F).\label{mon2}
\end{align}

Let $X$ be a compact Hausdorff space equipped with a free action of a compact Hausdorff  group~$G$.
Next, let
$\mathrm{FRep}(G)$ denote the category of representations of $G$ on finite-dimensional  complex vector spaces.
 Due to the freeness asumption, we can define the functor
\[
\mathrm{FRep}(G)\longrightarrow \mathrm{Vect}(X/G),\quad
V\longmapsto X\underset{G}{\times}V,
\]
preserving $\oplus$ and~$\otimes$: 
\begin{align}
X\underset{G}{\times}(V\oplus W) &= (X\underset{G}{\times}V)\:\oplus\:(X\underset{G}{\times}W),\\
X\underset{G}{\times}(V\otimes W) &= (X\underset{G}{\times}V)\:\otimes\:(X\underset{G}{\times}W).\label{mon1}
\end{align}
Combining the functor $\Gamma$ with the functor $X\times_G$ yields the functor
\[
\mathrm{FRep}(G)\longrightarrow \mathrm{FProj}(C(X/G)),\quad V\longmapsto \Gamma(X\underset{G}{\times}V).
\]
Furthermore, note that the $C(X/G)$-module $C_G(X,V)$ of all continuous $G$-equivariant functions from $X$ to $V$
is naturally isomorphic with $\Gamma(X\times_GV)$. Here $G$-equivariance means
\[
\forall\;x\in X,\, g\in G:\;f(xg)=\varrho(g^{-1})(f(x)),\qquad \varrho:\;G\longrightarrow GL(V).
\]
Hence
we can replace the above $\otimes$-preserving functor with the $\otimes$-preserving functor
 \[
\mathrm{FRep}(G)\longrightarrow \mathrm{FProj}(C(X/G)),\quad V\longmapsto C_G(X,V).
\]

The following  elementary observation is key in 
translating from the topological to the algebraic setting. 
\begin{lemma}\label{LemmaA}
Let $X$ be a compact Hausdorff space equipped with an action of a compact Hausdorff  group~$G$, and let
$V$ be a finite-dimensional representation of~$G$. Then the evident identification 
$C(X,V)=C(X)\otimes V$ determines
an equivalence of tensor functors:
\[
C_G(X,V) =\mathcal{P}_{G}(X)\Box V.
\]
\end{lemma}
\begin{proof}
Let $\{e_i\}_{i=1}^n$ be  a basis of $V$ and $\{e^i\}_{i=1}^n$ be the dual basis of~$V^*$.
Given $f\in C(X,V)$, we note that 
\begin{gather}
\sum_{i=1}^n (e^i\circ f)\otimes e_i\in\mathcal{P}_{G}(X)\Box V\nonumber\\
\Updownarrow\nonumber\\
\sum_{i=1}^n \delta(e^i\circ f)\otimes e_i=\sum_{i=1}^n (e^i\circ f)\otimes {}_V\Delta(e_i)\nonumber\\
\Updownarrow\nonumber\\
\forall\;x\in X,\, g\in G:\;f(xg)=\varrho(g^{-1})(f(x)).
\end{gather}
Indeed, the second equivalence is an immediate consquence of the definitions of coactions $\delta$ and~$\varrho$ (see \eqref{dualco}
and~\eqref{repcorep}). The first equivalence follows directly from the definition of cotensor product (see~\eqref{pwcot}) and the fact that
\[
\sum_{i=1}^n (e^i\circ f)\otimes {}_V\Delta(e_i)\in C(X)\otimes\mathcal{O}(G)\otimes V.
\]
Thus the evident identification yields $C_G(X,V) =\mathcal{P}_{G}(X)\Box V$.

Finally, let $\beta$ be the map defined in Theorem~\ref{4thm}, and let
\begin{align}
\text{diag}\colon C_G(X,V)\underset{C(X/G)}{\otimes}C_G(X,W)&\longrightarrow C_G(X,V\otimes W),
\nonumber\\
\text{diag}\colon f_1\otimes f_2&\longmapsto \big(x\mapsto f_1(x)\otimes f_2(x)\big).
\end{align}
The commutativity of the  diagram
\[
\xymatrix{
C_G(X,V)\underset{C(X/G)}{\otimes}C_G(X,W)\ar[r]^{\phantom{xxxxxx}\text{diag}}\ar[d]^{}&C_G(X,V\otimes W) \ar[d]^{}\\
(\mathcal{P}_G(X)\Box V)\underset{C(X/G)}{\otimes}(\mathcal{P}_G(X)\Box W)\ar[r]^{\phantom{xxxxxxxx}\beta}&\mathcal{P}_G(X)\Box (V\otimes W)
}
\]
proves that the identification  $C_G(X,V) =\mathcal{P}_{G}(X)\Box V$ defines an equivalence of tensor functors.
\end{proof}

Assume now that the action of $G$ on $X$ is free. Then, by the Serre-Swan theorem, the functor
$\Gamma(X\times_G\;)$ is strongly monoidal. Since it is equivalent as a tensor functor to $C_G(X,\;)$, we conclude from
Lemma~\ref{LemmaA} that the cotensor product functor
\[
\mathrm{FRep}(G)\longrightarrow \mathrm{FProj}(C(X/G)),\quad V\longmapsto \mathcal{P}_{G}(X)\Box V
\]
is also strongly monoidal.

Next, since $\mathcal{O}(G)$ is cosemisimple, any $\mathcal{O}(G)$-comodule is a purely algebraic direct 
sum
of finite-dimensional comodules. Furthermore, as the cotensor product is defined as the kernel of a linear map,
it commutes with such direct sums. As it is also clear that the map $\beta$ commutes with such direct sums,
we infer that the extended cotensor product functor
\[
\mathrm{FRep}^\oplus(G)\longrightarrow \mathrm{FProj}^\oplus(C(X/G)),\quad V\longmapsto 
\mathcal{P}_{G}(X)\Box V,
\]
is  strongly monoidal. Here $\mathrm{FProj}^\oplus(C(X/G))$ is the category of projective modules over 
$C(X/G)$ that are purely algebraic direct
sums of finitely generated projective $C(X/G)$-modules,
and $\mathrm{FRep}^\oplus(G)$ is the category of representations of $G$ defined above~\eqref{repcorep}.
(One can think of these categories as the ind-completions in the sense of~\cite[Section~8.2]{agv72}.)
Combining this with Corollary~\ref{4cor} allows us 
to conclude the proof of the implication 
``PWG-condition\hspace{2mm}$\Longleftarrow$\hspace{2mm}freeness''.

\section{Vector-bundle interpretation}

We now give a vector-bundle interpretation of the proof of the preceding section. To this end,
 we need to extend the functor $C_G(X,\;)$ to the  category $\mathrm{FRep}^\oplus(G)$, 
which includes the representation~$\mathcal{O}(G)$.
Let
$V$ be a purely algebraic direct sum of finite-dimensional representations of~$G$. We topologize $V$
as the direct limit of its finite-dimensional subspaces, and denote by $C(X,V)$ the space of
all continuous maps from $X$
to~$V$. An elementary topological argument shows that the image of any continuous map from $X$ to $V$ is 
contained in a finite-dimensional subspace of~$V$. Therefore, 
Lemma~\ref{LemmaA} generalizes to:
\begin{corollary}\label{LemmaB}
Let $V$ be an object in the category $\mathrm{FRep}^\oplus(G)$. Then the evident identification 
\mbox{$C(X,V)=C(X)\otimes V$} determines
an equivalence of tensor functors:
$$
C_G(X,V) =\mathcal{P}_{G}(X)\Box V.
$$
\end{corollary}
\vspace*{-4.5mm}
\noindent
Taking $V=\mathcal{O}(G)$ topologized with the direct limit topology, we immediately obtain the following 
presentation of the Peter-Weyl algebra:
\[\label{322}
C_G(X,\mathcal{O}(G)) = \mathcal{P}_{G}(X)\Box\mathcal{O}(G)\cong\mathcal{P}_{G}(X).
\]

Assume now that the action of $G$ on~$X$ is free. Then $X\times_G\mathcal{O}(G)$ is a vector bundle
in the sense that it is a direct sum of ordinary (i.e.\ with finite-dimensional fibers) vector bundles,
and 
\[
\Gamma(X\times_G\mathcal{O}(G))\cong C_G(X,\mathcal{O}(G))\cong\mathcal{P}_{G}(X).
\] 
Moreover, arguing as for the cotensor product functor, we conclude that the functor
\[
\mathrm{FRep}^\oplus(G)\longrightarrow \mathrm{FProj}^\oplus(C(X/G)),\quad V\longmapsto 
C_G(X,V),
\]
is strongly monoidal. Hence, taking advantage of \eqref{322}, we obtain
\[\label{325}
C_G(X,\mathcal{O}(G)\otimes\mathcal{O}(G)) \cong \mathcal{P}_{G}(X)
\underset{C(X/G)}{\otimes}\mathcal{P}_{G}(X).
\]
Next,
denote by $\mathcal{O}(G)^{\mathrm{trivial}}$ 
the vector space $\mathcal{O}(G)$ with the trivial action of $G$, 
i.e.\ every $g\in G$ is acting by the identity map of $\mathcal{O}(G)$.
Then,  as before, we obtain
\[\label{326}
C_G(X,\mathcal{O}(G)\otimes\mathcal{O}(G)^{\mathrm{trivial}}) 
\cong \mathcal{P}_{G}(X)
\underset{C(X/G)}{\otimes}C(X/G)\otimes\mathcal{O}(G)
\cong \mathcal{P}_{G}(X)\otimes\mathcal{O}(G).
\]

\begin{lemma}\label{replem}
The $G$-equivariant homeomorphism
$$
W\colon G\times G^{\mathrm{trivial}}\longrightarrow G\times G,\quad W((g,g')):=(g,gg'),
$$
gives an isomorphism of representations of~$G$
$$
\mathcal{O}(G)\otimes\mathcal{O}(G)^{\mathrm{trivial}}\cong \mathcal{O}(G)\otimes\mathcal{O}(G).
$$
Here $G\times G^{\mathrm{trivial}}$ and $G\times G$ are right $G$-spaces  via the formulas
$$
(g,g')h:=(h^{-1}g,g')\qquad\text{and}\qquad (g,g')h:=(h^{-1}g,h^{-1}g'),
$$
respectively.
\end{lemma}
\begin{proof}
Since $\mathcal{O}(G)$ is a Hopf algebra, the pullback of $W$ restricts and corestricts to
\[
W^*\colon \mathcal{O}(G)\otimes\mathcal{O}(G)\longrightarrow
\mathcal{O}(G)\otimes\mathcal{O}(G)^{\mathrm{trivial}}.
\]
Taking into account \eqref{015} and \eqref{repcorep}, 
we infer that $W^*$ is the required intertwining operator.
\end{proof}

Combining Lemma~\ref{replem} with \eqref{325} and \eqref{326} gives
\[
\mathcal{P}_{G}(X)\underset{C(X/G)}{\otimes}\mathcal{P}_{G}(X)\cong
\mathcal{P}_{G}(X)\otimes\mathcal{O}(G).
\]
Finally, to see that this isomorphism 
is indeed   the canonical map, 
we  explicitly put together
all identifications used on the way. First, we observe that, since  the isomorphism 
\[
\mathcal{P}_{G}(X)\longrightarrow
\mathcal{P}_{G}(X)\Box\mathcal{O}(G)
\]
 is given by the coaction $\delta$, the identification \eqref{322}
is implemented by the  maps 
\begin{gather}
\xymatrix@=20mm{
\P_G(X)\ar@<1ex>[r]^{E} &
C_G(X,\mathcal{O}(G)),\ar@<1ex>[l]^{F}
}
\nonumber\\
\big(E(f)(x)\big)(g):=f(xg),\quad F(\alpha)(x):=\alpha(x)(e),\quad
E\circ F=\mathrm{id},\quad F\circ E=\mathrm{id}.
\end{gather}
We can now easily check that the following composition of isomorphisms
\begin{align*}
&\P_G(X)\!\!\!{\underset{C(X/G)}{\otimes}}\!\!\!\P_G(X)
\overset{E\otimes E}{\longrightarrow}
C_G(X,\mathcal{O}(G))
\!\!\!{\underset{C(X/G)}{\otimes}}\!\!\!
C_G(X,\mathcal{O}(G))
\overset{\mathrm{diag}}{\longrightarrow}
C_G\big(X,\mathcal{O}(G)\otimes\mathcal{O}(G)\big)
\overset{W^*\!\circ}{\longrightarrow}
\nonumber\\ &
C_{G}
\big(X,\mathcal{O}(G)\otimes\mathcal{O}(G)^{\mathrm{trivial}}\big)
\overset{\text{$\sum_i(\mathrm{id}\otimes e^i)\otimes e_i$}}{\longrightarrow}
C_{G}
(X,\mathcal{O}(G))\otimes\mathcal{O}(G)
\overset{F\otimes\mathrm{id}}{\longrightarrow}
\P_G(X)\otimes\mathcal{O}(G)
\end{align*}
is the canonical map, as desired.

\section{Application: fields of $C^*$-free actions}
\setcounter{equation}{0}

Let $A$ be a unital $C^*$-algebra with center $Z(A)$, let $X$ be a compact Hausdorff space and let $\theta:C(X)\rightarrow Z(A)$ be a unital inclusion. The triple $(A,C(X),\theta)$ is called a unital \emph{$C(X)$-algebra} (\cite[p.~154]{k-g88}). In the following, we simply consider $C(X)$ as a subalgebra of $A$. For $x\in X$, let $J_x$ be the closed 2-sided ideal in $A$ generated by the functions $f\in C(X)$ that vanish at $x$. Then we have quotient $C^*$-algebras $A_x = A/J_x$ with natural projection maps $\pi_x:A\rightarrow A_x$, and the triple $(X,A,\pi_x)$ is a \emph{field of $C^*$-algebras}. For any $a\in A$, the map $n_x: X \rightarrow \R, x \mapsto \|\pi_x(a)\|$ is upper semi-continuous \cite[Theorem~2.4]{dg83} 
(see also \cite[Proposition 1.2]{r-ma89}). If the latter map is continuous, the field is called continuous, but this property will not be necessary to assume for our purposes.


\begin{lemma}\label{LemDelx}
 Let $X$ be a compact Hausdorff space, $A$ a unital $C(X)$-algebra, and $(H,\Delta)$ a compact quantum group acting on $A$ via $\delta:A\rightarrow A\mten H$. Assume that $C(X)\subseteq A^{\co H}$. Then for each $x\in X$ there exists a unique coaction $\delta_x:A_x\rightarrow A_x\mten H$ such that for all $a\in A$ 
\begin{equation}\label{DefRed}
 \delta_x(\pi_x(a)) = (\pi_x\otimes \id)(\delta(a)).
\end{equation}
\end{lemma}
\begin{proof} 
Let $x\in X$ and $f\in C(X)$ with $f(x)=0$. As $\delta(f) = f\otimes 1$ by assumption, it follows that 
$(\pi_x\otimes\id)(\delta(f)) = 0$. Hence $(\pi_x\otimes \id)(\delta(a))= 0$ for $a\in J_x$, so that $\delta_x$ 
can be defined by~\eqref{DefRed}. It is straightforward to check that each $\delta_x$ satisfies the 
coassociativity and counitality conditions.

Finally, to see that $\delta_x$ is injective, assume that $\delta_x(\pi_x(a))=0$. 
Then $(\pi_x\otimes \id)(\delta(a))=0$, whence $(\id\otimes \omega)(\delta(a)) \in J_x$ 
for all $\omega \in A^*$. In particular, if $(g_{\alpha})_{\alpha}$ is a bounded positive approximate unit for 
$C_0(X\setminus\{x\})$, then 
\[
g_{\alpha}(\id\otimes \omega)(\delta(a))\overset{\mathrm{norm}}{\underset{\alpha}{\longrightarrow}}
 (\id\otimes \omega)(\delta(a)).
\] 
Hence we obtain
 \[
(g_{\alpha}\otimes 1)\delta(a)\overset{\mathrm{weakly}}{\underset{\alpha}{\longrightarrow}} \delta(a).
\] 
However, as $(g_{\alpha}\otimes 1)\delta(a) = \delta(g_{\alpha}a)$ and  $\delta$ is injective, we find that 
\[
g_{\alpha}a\overset{\mathrm{weakly}}{\underset{\alpha}{\longrightarrow}} a.
\] 
Consequently,  $\pi_x(a)=0$, and we conclude that $\delta_x$ is injective.
\end{proof}
\vspace*{1mm}

\begin{theorem}\label{propBundle}  
Let $X$ be a compact Hausdorff space, $A$ a unital $C(X)$-algebra, and $(H,\Delta)$ a compact quantum 
group acting on $A$ via $\delta:A\rightarrow A\mten H$. Assume that $C(X)\subseteq A^{\co H}$. Then,  the 
coaction $\delta$ is $C^*$-free
\emph{if and only if} the coactions $\delta_x$ are $C^*$-free for each $x\in X$.
\end{theorem}
\vspace*{-5mm}
\begin{proof} 
First note that $A\mten H$ is again a $C(X)$-algebra in a natural way. 
We will denote the quotient $(A\mten H)/(J_x\mten H)$ by $A_x\otimes_x H$. 
This will be a $C^*$-completion of the algebraic tensor product algebra $A_x\otimes H$ (not necessarily the 
minimal one). We will 
denote the quotient map at $x$ by $\pi_x\otimes_x \id:A\mten H \rightarrow A_x\otimes_x H$.

The implication ``$\delta$ is $C^*$-free $\Longrightarrow$ the coactions $\delta_x$ are $C^*$-free for each 
$x\in X$" follows immediately
from the commutativity of the diagram
\begin{displaymath} \xymatrix{
A\otimes A\ar[d]_{\pi_x\otimes\pi_x} \ar[r]^{can}       &       A\underset{\mathrm{min}}{\otimes} H\ar[d]^{\pi_x\otimes\id}  \\
A_x\otimes A_x \ar[r] &  A_x\underset{x}{\otimes} H\,. }
\end{displaymath}
Here the upper horizontal arrow is given by the  formula $a\otimes a'\mapsto (a\otimes 1)\delta(a')$, and the lower horizontal
arrow is given by $a\otimes a'\mapsto (a\otimes 1)\delta_x(a')$.

Assume now that each $\delta_x$ is $C^*$-free. Fix $\varepsilon>0$, and choose $h\in \cO(H)$.
 By Theorem \ref{theoGal}, for each $x\in X$ we can find an element 
$z_x \in (A\otimes 1)\delta(A)$ 
such that $(\pi_x\otimes_x \id)(z_x) = 1\otimes h$ in $A_x\otimes_x H$.  Consider the function
 \[
f_x: X\ni y \longmapsto \|(\pi_y\otimes_y\id)(z_x-1\otimes h)\|= \|(\pi_y\otimes_y\id)(z_x)-1\otimes h\|\in \R.
\] 
As the norm on the field $y\mapsto A_y\otimes_yH$ is upper semi-continuous, 
the function $y\mapsto f_x(y)$ is upper semi-continuous. 
Since $f_x(x)=0$, we can find an open neighborhood $U_x$ of $x$ such that for all $y\in U_x$ 
\[
f_x(y)= \|(\pi_y\otimes_y \id)(z_x)- 1\otimes h\|_{A_y\otimes_yH} <\varepsilon.
\]

Let $\{f_i\}_i$ be a partition of unity subordinate to a finite subcover $\{U_{x_i}\}_i$. An easy estimate shows 
that for $z := \sum_i (f_i\otimes 1)z_{x_i}$ and all $y\in X$ 
\[
\|(\pi_y\otimes_y \id)(z -1\otimes h)\|_{A_y\otimes_y H} < \varepsilon.
\] 
Taking the supremum over all $y$, we conclude by 
\cite[Theorem~2.4]{dg83} and the compactness of $X$ that  $\|z-1\otimes h\|< \varepsilon$. 
Hence $(A\otimes 1)\delta(A)$ 
is dense in $A\otimes H$, i.e.\ the coaction $\delta$ is $C^*$-free.
\end{proof}

Combining Theorem~\ref{theoGal} and Theorem~\ref{propBundle}, we obtain:
\begin{corollary}\label{application}
Let $X$ be a compact Hausdorff space, $A$ a unital $C(X)$-algebra, and $(H,\Delta)$ a compact quantum 
group acting on $A$ via $\delta:A\rightarrow A\mten H$. Assume that $C(X)\subseteq A^{\co H}$. Then,  the 
coaction $\delta$ satisfies
the PWG-condition
\emph{if and only if} the coactions $\delta_x$ satisfy the PWG-condition for each $x\in X$.
\end{corollary}

As a particular case we consider: 
\begin{definition}[cf.~\cite{dhh}]\label{gjoin}
Let $(H,\Delta)$ be a compact quantum group acting on a unital \mbox{$C^*$-algebra} $A$ via 
$\delta:A\rightarrow A\otimes_{\mathrm{min}}H$.
We call the unital $C^*$-algebra
\begin{equation}
A \underset{\delta}{\circledast} H := 
\left\{f\in C\big([0,1],A \underset{\mathrm{min}}{\otimes} H\big) \,
\big|\, f(0) \in \mathbb{C}\otimes H,\; f(1)\in\delta(A)\right\}
\end{equation}
the \emph{equivariant noncommutative join} of $A$ and $H$. 
\end{definition}

The $C^*$-algebra $A\circledast_\delta H$ 
is obviously a $C(\lbrack 0,1\rbrack)$-algebra with $(A\circledast_\delta H)_x \cong A\mten H$ for 
$x\in (0,1)$, $(A\circledast_\delta H)_0 \cong H$ and $(A\circledast_\delta H)_1 \cong A$. 
We identify $A\circledast_\delta H$ as a subalgebra of 
$C(\lbrack 0,1\rbrack)\mten A\mten H$.
 The following lemma shows that $A\circledast_\delta H$ carries a natural $H$-coaction.
\vspace*{0mm}
\begin{lemma}\label{lemDef} 
The compact quantum group $(H,\Delta)$ acts on the unital $C^*$-algebra $A \circledast_\delta H$ via
\[
\delta_{A\underset{\delta}{\circledast} H}\colon A \underset{\delta}{\circledast} H\ni f\longmapsto 
(\id\otimes\id\otimes\Delta)\circ f\in 
(A \underset{\delta}{\circledast} H)\otimes_{\mathrm{min}}H.
\]
\end{lemma}
\vspace*{-3mm}
\begin{proof} 
Note that $\delta_{A\circledast_\delta H}$ is the restriction of $(\id\otimes \id\otimes \Delta)$ to 
$A\circledast_\delta H$.  
Let us first show that the range of $\delta_H$ is contained in $(A\circledast_\delta H)\mten H$.

Consider an element $F\in A\circledast_\delta H$ as an $A\mten H$-valued function on
 $\lbrack 0,1\rbrack$. Since $F$ is 
uniformly continuous and $\cP_H(A)$ is dense in $A$ by \cite[Theorem~1.5.1]{p-p95} and
 \cite[Proposition~2.2]{s-pm11}, an elementary partition of unity argument shows that $F$ can be 
approximated by a finite sum of functions of three kinds:
\begin{enumerate}
\item
$F_1:[0,1]\ni t\mapsto \xi_0(t)(1\otimes h)\in \mathbb{C}\otimes \cO(H)$, where
$\xi_0\in C([0,1],[0,1])$, $\xi_0(1)=0$, and $h$ is a fixed element of~$\cO(H)$.
\item
$F_2:[0,1]\ni t\mapsto \xi(t)(a\otimes h)\in \cP_H(A)\otimes_{\mathrm{alg}} \cO(H)$, where
$\xi\in C([0,1],[0,1])$ with $\xi(0)=\xi(1)=0$, and $a$ and $h$ are fixed elements of $\cP_H(A)$ and $\cO(H)$ 
respectively.
\item
$F_3:[0,1]\ni t\mapsto \xi_1(t)\delta(a)\in \delta(\cP_H(A))$,  where
$\xi_1\in C([0,1],[0,1])$, $\xi_1(0)=0$, and $a$ is a fixed element of $\cP_H(A)$.
\end{enumerate}
It is clear that $\delta_{A\circledast_\delta H}(F_i) \in C(\lbrack 0,1\rbrack, A\mten H)\otimes_{\rm{alg}} 
\cO(H)$ for all $i$. Let $\omega$ be a functional on~$\cO(H)$. Then 
$(\id\otimes \omega)(\delta_{A\circledast_\delta H}(F_i)) \in A\circledast_\delta H$ for all~$i$. 
This implies that  $\delta_{A\circledast_\delta H}(F_i) \in (A\circledast_\delta H)\otimes_{\rm{alg}} H$
 for all~$i$. It follows from the continuity of $\delta_{A\circledast_\delta H}$ that 
$\delta_{A\circledast_\delta H}(F) \in (A\circledast_\delta H)\mten H$. 
Hence $\delta_{A\circledast_\delta H}$ has range in $(A\circledast_\delta H)\otimes_{\rm{min}} H$.

The coassociativity of $\delta_{A\circledast_\delta H}$ is immediate from the coassociativity of $\Delta$. 
The counitality 
condition follows from the same approximation argument as above.
\end{proof}

\begin{corollary} 
The coaction $\delta_{A\circledast_\delta H}:
A\circledast_\delta H\to (A\circledast_\delta H)\mten H$  
is $C^*$-free.
\end{corollary}
\vspace*{-8mm}
\begin{proof} The $C^*$-algebra $A\circledast_\delta H$ is a unital $C(\lbrack 0,1\rbrack)$-algebra 
with $C(\lbrack 0,1\rbrack) \in (A\circledast_\delta H)^{\co H}$. 
With the notation of Lemma~\ref{LemDelx}, we have: 
\vspace*{-2mm}
\begin{enumerate}
\item $((A\circledast_\delta H)_0,\delta_0) \cong (H,\Delta)$,
\item $((A\circledast_\delta H)_x,\delta_x) \cong (A\mten H,\id\otimes \Delta)$ for $x\in (0,1)$,
\item $((A\circledast_\delta H)_1,\delta_1)\cong (A,\delta)$.
\end{enumerate}
\vspace*{-2mm}
As each of the above actions are $C^*$-free, we infer from Theorem~\ref{propBundle} that 
$\delta_{A\circledast_\delta H}$ is $C^*$-free. 
Alternatively, one can use a direct 
approximation argument as in Lemma~\ref{lemDef}.
\end{proof}

\appendix

\renewcommand{\thesection}{A}

\section*{Appendix: Finite Galois coverings}
\setcounter{section}{0}

Let $\pi\colon X\rightarrow Y$ be a \emph{covering map} of topological spaces. As usual, this means that 
given any $y\in Y$ there exists an
open set $U$ in $Y$ with $y\in U$ such that $\pi^{-1}(U)$ is a disjoint union of open sets each of which $\pi$ 
maps homeomorphically
onto~$U$. A \emph{deck transformation} is a homeomorphism $h\colon X\rightarrow X$ with $\pi\circ h=\pi$.
\vspace*{1mm}
\begin{proposition}\label{propo}
Let $X$ and $Y$ be compact Hausdorff topological spaces. Let $\pi\colon
X\rightarrow Y$ be a covering map, and let $\Gamma$ be
the group of deck transformations of this covering. Assume that $\Gamma$ is finite. Then
 $X$ is a principal $\Gamma$-bundle over $Y$ if and only if the canonical map
\begin{align*}
can \colon C(X){\underset{C(Y)}{\otimes}}C(X)&\longrightarrow C(X)\otimes C(\Gamma) \\
can \colon f_1\otimes f_2&\longmapsto (f_1\otimes 1)\delta(f_2)
\end{align*}
is an isomorphism. Here $\delta$ is given by~\eqref{dualco}.
\end{proposition}
\vspace*{-5mm}
\begin{proof}
If $X$ is a principal $\Gamma$-bundle over $Y$, then $C(Y)=C(X/\Gamma)=C(X)^{co C(\Gamma)}$ and, by 
\eqref{classcls},
$can$ is surjective. Furthermore, since $C(\Gamma)$ is cosemisimple, 
by the result of H.-J.~Schneider \cite[Theorem~I]{s-hj90}, the surjectivity of $can$ implies its bijectivity.

Assume now that $can$ is bijective.
The local triviality assumption in the definition of a covering map implies that for any continuous function $f$ on 
$X$ one has
a continuous function $\Theta(f)$ on $Y$ given by the formula
\[
(\Theta(f))(y):=\frac{1}{\# \pi^{-1}(y)}\sum_{x\in\pi^{-1}(y)}f(x)\,.
\]
Note that the fibres are finite due to the compactness of $X$. Also, 
one immediately sees that $\Theta$ is a unital $C(Y)$-linear map from $C(X)$ to $C(Y)$. Now it follows from the bijectivity of $can$
and \cite[Lemma~1.7]{dhs99} that $C(Y)=C(X)^{co C(\Gamma)}=C(X/\Gamma)$. Hence the fibres of the covering 
$\pi\colon X\rightarrow Y$ are the orbits of~$\Gamma$. Finally, the freeness of the action of $\Gamma$ on $X$ follows
from the surjectivity of $can$ and~\eqref{classcls}.
\end{proof}

If $X$ is connected, then it is always the case that
the group of deck transformations $\Gamma$ is finite and that the action of $\Gamma$
on $X$ is free. The issue is then whether or not the action of $\Gamma$ on each fiber of $\pi$ is transitive. Thus we conclude
from Proposition~\ref{propo}:
\begin{corollary}
Let $X$ and $Y$ be connected compact Hausdorff topological spaces, and  let $\pi\colon X\rightarrow Y$ be a
covering map.  Denote by $\Gamma$ the group of deck transformations. Then the action of $\Gamma$ on 
each fiber of $\pi$
is transitive if and only if the canonical map
\[
can \colon C(X){\underset{C(Y)}{\otimes}}C(X)\longrightarrow C(X)\otimes C(\Gamma)
\]
is an isomorphism.
\end{corollary}
\vspace*{0mm}

\begin{remark}{\rm
To make the proof of Proposition~\ref{propo} more self-contained, let us unravel the crux of the argument 
proving \cite[Lemma~1.7]{dhs99}.
We know that $C(Y)\subseteq C(X/\Gamma)$, and we need to prove the equality. To this end, let us take any 
$f\in C(X/\Gamma)$.
Then, since $can (1\otimes f)=can(f\otimes 1)$, it follows from the bijectivity of $can$ that 
$1\otimes f=f\otimes 1\in C(X)\otimes_{C(Y)}C(X)$. Applying $\Theta\ot \id$ to this equality yields 
$f=\Theta(f)\in C(Y)$.
}\end{remark}
\vspace*{0mm}

\begin{remark}{\rm
An alternative proof of Proposition~\ref{propo} is as follows. Consider the commutative diagram
\begin{displaymath} \xymatrix{
C(X){\underset{C(Y)}{\otimes}}C(X)\ar[d] \ar[r]^{can}       &       C(X)\otimes C(\Gamma)\ar[d]  \\
C(X{\underset{Y}{\times}}X) \ar[r] &  C(X\times\Gamma) }
\end{displaymath}
in which each vertical arrow is the evident map and the lower horizontal arrow is the \mbox{$\ast$-ho}momor- phism
resulting from the map of topological spaces
\[
X\times\Gamma\longrightarrow X\underset{Y}{\times}X,\qquad (x, \gamma)\mapsto (x, x\gamma).
\]
Note that $X$ is a (locally trivial) principal $\Gamma$ bundle on $Y$ if and only if this map of topological spaces is
a homeomorphism --- which is equivalent to bijectivity of the lower horizontal arrow.

Hence to prove Proposition~\ref{propo}, it will suffice to prove that the two vertical arrows are isomorphisms.
The right vertical arrow is an isomorphism because $\Gamma$ is a finite group, so $C(\Gamma)$ is a
finite dimensional vector space over the complex numbers $\mathbb{C}$.

For the left vertical arrow, let $E$ be the vector bundle on $Y$ whose fiber at $y\in Y$ is Map($\pi^{-1}(y), \mathbb{C}$),
i.e.\ is the set of all set-theoretic maps from $\pi^{-1}(y)$ to $\mathbb{C}$. As $\pi^{-1}(y)$ is a discrete subset of the
compact Hausdorff space $X$, it is finite.
Let $\mathcal{S}(E)$ be the algebra consisting of all the continuous
sections of $E$. Then $\mathcal{S}(E)=C(X)$.

Similarly, define
\[
\pi^{(2)}\colon X{\underset{Y}{\times}}X \longrightarrow Y\quad\text{by}\quad
\pi^{(2)}\colon (x_1, x_2)\longmapsto \pi(x_1)=\pi(x_2).
\]
Let $F$ be the vector bundle on $Y$ whose fiber at $y\in Y$ is Map($(\pi^{(2)})^{-1}(y), \mathbb{C}$), i.e.\ 
is the set of all
set-theoretic maps from $(\pi^{(2)})^{-1}(y)$ to $\mathbb{C}$. Then 
$\mathcal{S}(F) =C( X\times_YX)$,
where $\mathcal{S}(F)$ is the algebra consisting of all the continuous sections of~$F$.
Since $F=E\otimes E$ as vector bundles on~$Y$,
we conclude $\mathcal{S}(F)=\mathcal{S}(E)\otimes_{C(Y)}\mathcal{S}(E)$, which proves bijectivity for the 
left vertical arrow. 
}\end{remark}
\vspace*{0mm}

\begin{example}{\rm
Without connectivity,  the group of deck transformations can be infinite. For example, let $Y$ be the
Cantor set and let 
$\pi\colon Y\times\{0, 1\}\rightarrow Y$ be the trivial twofold covering.
Let $U$ be a subset of $Y$ which is both open and closed. 
Define $\gamma_U\colon Y\times\{0, 1\}\rightarrow Y\times\{0, 1\}$ by 
\[ 
\gamma_U(y, t) = \left \{
                                                 \begin{array}{ll}
                                                 (y, t) & \text{for }\mbox{$y\notin U$}\\

                                                                      &              \\

                                                 (y, 1-t) & \text{for }\mbox{$y\in U$. }
                                                 \end{array}
                                                 \right. 
\]
Then $\gamma_U$ is a deck transformation and there are infinitely many~$\gamma_U$.
}\end{example}
\vspace*{0mm}

\begin{example}{\rm 
The following example is a three-fold cover $X$ of the one-point union of two circles~$Y$. Here the pre-image 
of the left circle
of the base space is the usual three-fold covering of the circle. The pre-image of the right circle of the base 
space is the
disjoint union of the usual two-fold covering of the circle and the one-fold covering of the circle.
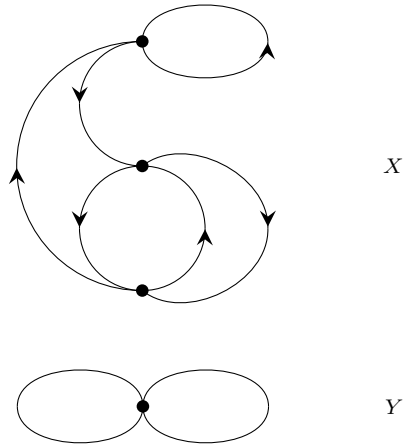
\begin{center}
\begin{tikzpicture}[x=2em, y=2em, >=stealth, auto, node distance=1,scale=1]
\tikzset{
  on each segment/.style={
    decorate,
    decoration={
      show path construction,
      moveto code={},
      lineto code={
        \path [#1]
        (\tikzinputsegmentfirst) -- (\tikzinputsegmentlast);
      },
      curveto code={
        \path [#1] (\tikzinputsegmentfirst)
        .. controls
        (\tikzinputsegmentsupporta) and (\tikzinputsegmentsupportb)
        ..
        (\tikzinputsegmentlast);
      },
      closepath code={
        \path [#1]
        (\tikzinputsegmentfirst) -- (\tikzinputsegmentlast);
      },
    },
  },
  mid arrow/.style={postaction={decorate,decoration={
        markings,
        mark=at position .5 with {\arrow[#1, ultra thick]{>}}
      }}},
  end arrow/.style={postaction={decorate,decoration={
        markings,
        mark=at position .99 with {\arrow[#1, ultra thick]{>}}
      }}},
}

\tikzstyle{every circle} = [radius=0.1, fill=black]
\tikzstyle{every node} = [font=\tiny]

\fill (0, 2) circle node {};
\fill (0, 0) circle node {};
\fill (0, -2) circle node {};

\draw (4, 0) node {$X$};

\path	
   (0,-2)edge	[bend left=45, end arrow]	(-2, 0)
   (-2,0)edge	[bend left=45]	(0, 2)
	(0, 2)	edge	[bend right=90, end arrow]	(2, 2)
	(2, 2)	edge	[bend right=90]	(0, 2)
	(0, 2)	edge	[bend right=45, end arrow]	(-1, 1)
  (-1, 1)edge	[bend right=45]	(0, 0)
	(0, 0)	edge	[bend left=65, end arrow]	(2,-1) 
	(2,-1)	edge	[bend left=65]	(0,-2)
	(0,-2)	edge	[bend right=45, end arrow]	(1,-1)
	(1,-1)	edge	[bend right=45]	(0,0)
	(0, 0)	edge	[bend right=45, end arrow]	(-1, -1)
 (-1, -1)edge	[bend right=45]	(0, -2)
;

\end{tikzpicture}

\begin{tikzpicture}[x=2em, y=2em, >=stealth, auto, node distance=1]
\tikzstyle{every circle} = [radius=0.1, fill=black]
\tikzstyle{every node} = [ font=\tiny]

\fill (0, 0) circle node {};

\draw (4, 0) node {$Y$};

\path	(-2,0)	edge	[bend left=90]	(0,0)
	(0,0)	edge	[bend right=90]	(2,0)
	(2,0)	edge	[bend right=90]	(0,0)
	(0,0)	edge	[bend left=90]	(-2,0);
\end{tikzpicture}
\end{center}
In this example,   the group of deck transformations is trivial. Indeed,
let $\gamma$ be a deck transformation. Consider $\gamma$ restricted to the pre-image of the right circle of 
the base space.
This pre-image has two connected components. Since $\gamma$  is a deck transformation of this pre-image, 
it must map
each  connected compenent to itself. 
This implies that $\gamma$ has a fixed point. Hence, as $X$ is connected, $\gamma=\mathrm{id}$.
In particular, this shows that the group of deck transformations
need not act transitively on fibers of a covering.
The canonical map is surjective but not injective.
}\end{example}

\noindent{\bf Acknowledgments.}
We thank Wojciech Szyma\'nski and Makoto Yamashita
for helpful and enlightening discussions. We are also very grateful to Jakub Szczepanik for assistance with 
\LaTeX graphics.
This work was partially supported by NCN grant 
2011/01/B/ST1/06474. Paul F.\ Baum was partially supported by NSF grant DMS 0701184,
and Kenny De Commer was partially supported by FWO grant G.0251.15N.

\end{document}